\documentclass[preprint,12pt]{elsarticle}

\usepackage{amssymb}
 \usepackage{graphicx}
\usepackage{epsfig}

\usepackage{amsthm}

\usepackage{latexsym}          

\usepackage{psfrag}             

\usepackage{amsmath}            

\usepackage{amsfonts,amssymb}   
\usepackage[usenames]{color}
\usepackage{latexsym}          
\usepackage{graphicx}             
\usepackage{psfrag}             

\usepackage{colordvi}
\usepackage{srcltx}

\newcommand{\sysn}{\left\{\begin{array}{rcl}}
\newcommand{\sysk}{\end{array}\right.}

\newtheorem{theorem}{Theorem}[section]

\theoremstyle{example}
\newtheorem{example}[theorem]{Example}

\newtheorem{proposition}[theorem]{Proposition}
\theoremstyle{definition}
\newtheorem{definition}[theorem]{Definition}
\newtheorem{remark}[theorem]{Remark}
\newtheorem{corollary}[theorem]{Corollary}

\journal{...}

\begin{document}

\begin{frontmatter}

\title{Compact condensations of Hausdorff spaces}

\author{Vitalii I. Belugin}
\ead{belug54@mail.ru}


\address{Krasovskii Institute of Mathematics and Mechanics, 620219, Yekaterinburg, Russia}

\author{Alexander V. Osipov}
\ead{OAB@list.ru}


\address{Krasovskii Institute of Mathematics and Mechanics, Ural Federal
 University, \\ Ural State University of Economics, 620219, Yekaterinburg, Russia}

\author{Evgenii G. Pytkeev}
\ead{pyt@imm.uran.ru}


\address{Krasovskii Institute of Mathematics and Mechanics, Ural Federal
 University, 620219, Yekaterinburg, Russia}

\begin{abstract}
 In this paper, we continue to study one of the classic problems in general topology raised by P.S. Alexandrov:
  when a Hausdorff space $X$ has a continuous bijection (a condensation) onto a compactum? We concentrate on the
situation when not only $X$ but also $X\setminus Y$ can be
condensed onto a compactum whenever the cardinality of $Y$ does
not exceed certain $\tau$.


\end{abstract}


\end{frontmatter}

\section*{Introduction}

   The question of when each space from class $\mathcal{A}$ admits a continuous bijection (such map
is called a condensation) onto some space from class $\mathcal{B}$
is one of the natural questions in general theory, the subject of
which is the study of relations between classes of spaces,
performed by various types of mappings.

 In 1937 S. Banach posed a problem which can be formulated
 equivalently: when can a metric space have a condensation onto a
 compact metric space? Independently, the following more general question
 is attributed to P.S. Alexandrov: when a Hausdorff
space $X$ has a condensation onto a compactum?

It is natural to call such spaces as {\it subcompact} spaces. One
of the first and strong results were obtained by M. Katetov
\cite{kat}: an $H$-closed Urysohn space is subcompact.

In the future, an active study of Hausdorff spaces which admit a
condensation onto a compactum was continued in the works of I.L.
Raukhvarger \cite{raut}, V.V. Proizvolov \cite{proiz}, A.S.
Parhomenko \cite{parh1,parh2}, Y.M. Smirnov \cite{smir}, N.
Hadzhiivanov \cite{hedg1}, V.K. Bel'nov \cite{bel1}, A.V.
Arhangel'skii \cite{arh1,arh2}, O. Pavlov \cite{arh2}, V.I.
Belugin \cite{Belug0,Belug1}, E.G. Pytkeev
\cite{osipyt,Pytkeev3,Pytkeev1,Pytkeev2}, W.Kulpa and M.
Turza$\acute{n}$ski \cite{kt}, H. Reiter \cite{rit}, W.W. Comfort,
A.W. Hager and J. van Mill \cite{chjm} and many other authors.

 The fact that $X\setminus Y$ can be condensed onto a compactum
 for every countable $Y$ was established for metrizable compacta
 by Raukhvarger \cite{raut}, for products of metrizable compacta
 by Proizvolov \cite{proiz}, for diadic compacta by Belugin
 \cite{Belug0}, for weakly diadic compacta (including polyadic and
 centered spaces) by Kulpa and Turzanski \cite{kt}, for zero-dimensional
 first countable compacta by Belugin
 \cite{Belug1}. On the other hand Ponomarev \cite{ArchPonom}
 proved that if we remove from the remainder $\omega^*=\beta
 \omega\setminus \omega$ of the $\check{C}$ech-Stone
 compactification of $\omega$, a countable subset $D$ then
 $\omega^*\setminus D$ has no condensation onto a compactum.

It is well known that any locally compact admits a condensation
onto a compactum (Parhomenko's Theorem) \cite{parh2}. It turned
out that condensations onto a compactum are relatively rare.
  At the same
time, the most promising way of research appeared to be the one
started in \cite{raut}. I.L. Raukhvarger proved that for any
metric compact space $X$ and $C\in [X]^{\omega}$, the space
$X\setminus C$ admits a condensation $P$ onto some compact space
$Y_C$. The condensation $P: X\setminus C \rightarrow Y_C$ is a
quotient map (projection) where the decomposition space $Y_C$ is
obtained from $X$ by identifying the points belonging to the same
member of the decomposition $\mathcal{D}=\{(c,a): c\in C, a\in
A\}\cup \{\{b\}: b\in X\setminus (C\cup A)\}$ for a countable set
$A\subset X\setminus C$. Clearly, the condensation $P$ can be
extended to a continuous map $\widetilde{P}: X \rightarrow Y_C$.
This method
 (continuous decompositions) is quite effective in the study
condensations onto compact spaces. For example, by the method of
continuous decomposition, it was proved that any weakly diadic is
a strictly $a$-space \cite{kt}.

 In this paper, we concentrate on the
situation when not only $X$ but also $X\setminus Y$ can be
condensed onto a compactum whenever the cardinality of $Y$ does
not exceed certain $\tau$.

\section{Main definitions and notation}

In this paper, all considered spaces are assumed to be Hausdorff
topological spaces. We use a quotient space and related concepts.

Let $Y$ be a set. For a space $X$ and a surjection $f : X
\rightarrow Y$ , let $\tau(f)=\{U\subset Y : f^{-1}(U)$ is open in
$X\}$. Then $\tau(f)$ is called the {\it quotient topology} for
$Y$ determined by $f$.  Let $X$ and $Y$ be topological spaces. Let
$f : X \rightarrow Y$ be a surjection. Then $f$ is called a {\it
quotient map} if the topology in $Y$ is exactly $\tau(f)$; that
is, $U$ is open in $Y$ if and only if $f^{-1}(U)$ is open in $X$.
The space $Y$ is called the {\it quotient space} of $X$ by $f$
\cite{enc}.

Let $X$ be a set. Let $\mathcal{D}$ be a {\it decomposition} of
$X$; that is, $\mathcal{D}$ is a cover of $X$ such that any two
distinct members are disjoint. Let $P : X \rightarrow \mathcal{D}$
be the {\it projection} (i.e., $P$ maps each point $x\in X$ to the
unique member of $\mathcal{D}$ containing $x$). Let
$X(\mathcal{D})$ be the space $\mathcal{D}$ having the quotient
topology determined by $P$ (i.e., $\mathcal{D}'\subset
\mathcal{D}$ is open in $X(\mathcal{D})$ if and only if
$P^{-1}(\mathcal{D}')$ is open in $X$). The space $X(\mathcal{D})$
is called the {\it decomposition space} of $X$ by $\mathcal{D}$.
Namely, the decomposition space $X(\mathcal{D})$ is obtained from
$X$ by identifying the points belonging to the same member of
$\mathcal{D}$, and a subset of $X(\mathcal{D})$ is open if and
only if its inverse image by the projection $P$ is open in $X$.
The set $A\subset X$ is called {\it saturated with respect to the
decomposition $\mathcal{D}$}, if $A$ is the union of some set of
elements of $\mathcal{D}$ (i.e. for any element $T\in \mathcal{D}$
if $T\bigcap A\neq \emptyset$, then $T\subseteq A$) \cite{enc}.

 A decomposition $\mathcal{D}$ of a space $X$ is called
{\it continuous} if for any $T\in \mathcal{D}$ and for any open
set $U\supseteq T$ there is an open saturated set $V$ such that
$T\subseteq V\subseteq U$ \cite{alex,mor}.

P.S. Alexandrov and H. Hopf proved that a decomposition
$\mathcal{D}$ of a space $X$ is continuous if and only if the
projection $P: X \rightarrow X(\mathcal{D})$ is a closed map
\cite{alexhopf}. Using this result, it is easy to prove that the
quotient space $X(\mathcal{D})$ of a compact Hausdorff space $X$
is a compact Hausdorff space if and only if the decomposition
$\mathcal{D}$ is continuous and consists of closed subsets of $X$.

In this paper, we  use the following notations: $\omega$ - the
first infinite ordinal,  $\omega_1$ - the first uncountable
ordinal, $\aleph_0$ - the first infinite cardinal number,
$\mathbb{Q}$, $\mathbb{N}$ and $\mathbb{R}$ are, as usual, the set
of rational, natural and real numbers, respectively. For an
arbitrary set $A$ and a cardinal number $\tau$, $[A]^{\leq\tau}$ (
$[A]^{<\tau}$) will denote the set of all subsets of the set $A$
of the cardinality $\leq\tau$ ($<\tau$). A space $X$ is {\it
non-trivial} provided that $|X|>1$.

 In \cite{iliad}(Theorem 10) S. Iliadis proved the following theorem.

\begin{theorem}(S. Iliadis)\label{iliad} Let $\gamma X$ be a
extremally disconnected compactification of a space $X$ such that
the remainder $X^*=\gamma X\setminus X$ of cardinality
$<2^{\mathfrak{c}}$. Then $X$ is not subcompact space.
\end{theorem}

We use the particular case of this theorem: Let $X$ be an
extremally disconnected compact space without isolated points and
$C\in [X]^{\omega}$. Then $X\setminus C$ is not subcompact space.

\medskip

Let $\tau$ be an infinite cardinal number.

$\bullet$ A compact space $X$ is called an {\it $a_{\tau}$-space}
provided that for each $C\in [X]^{\leq\tau}$ there is a
condensation from $X\setminus C$ onto a compactum \cite{bop}.

In particular, for $\tau=\aleph_0$, a compact space $X$ is called
{\it an $a$-space} \cite{Belug1}.
\medskip

$\bullet$ A compact space $X$ is called an {\it strictly
$a_{\tau}$-space} provided that for each $C\in [X]^{\leq\tau}$
there is a condensation $f$ from $X\setminus C$ onto a compact
space $Y$ that $f$ can be extended to a continuous map
$\widetilde{f}: X \rightarrow Y$ \cite{bop}.

\medskip

Note that any strictly $a_{\tau}$-space is an $a_{\tau}$-space and
is a strictly $a_{\eta}$-space for any $\eta\leq \tau$. For
$\tau=\aleph_0$, a compact space $X$ is called a {\it strictly
$a$-space} \cite{Belug1}.

\medskip

 A natural extension of the classes of $a_{\tau}$- and strictly
$a_{\tau}$-spaces are classes of (strictly) $a_{\tau}$-subcompact
and almost (strictly) $a_{\tau}$-subcompact spaces.

\medskip

\begin{definition} Let $\tau$ be an infinite cardinal number.

$\bullet$ A space $X$ is called {\it $a_{\tau}$-subcompact}
provided that for each $C\in [X]^{\leq\tau}$ there is a
condensation from $X\setminus C$ onto an $a_{\tau}$-space.

$\bullet$ A space $X$ is called {\it strictly
$a_{\tau}$-subcompact} provided that for each $C\in
[X]^{\leq\tau}$ there is a condensation $f$ from $X\setminus C$
onto a strictly $a_{\tau}$-space $Y$ that $f$ can be extended to a
continuous map $\widetilde{f}: X \rightarrow Y$.

$\bullet$ A space $X$ is called {\it almost $a_{\tau}$-subcompact}
provided that for each $C\in [X]^{\leq\tau}$ there is a
condensation from $X\setminus C$ onto a compact space.

$\bullet$ A space $X$ is called {\it almost strictly
$a_{\tau}$-subcompact} provided that for each $C\in
[X]^{\leq\tau}$ there is a condensation $f$ from $X\setminus C$
onto a compact space $Y$ that $f$ can be extended to a continuous
map $\widetilde{f}: X \rightarrow Y$.
\end{definition}

Note that for an arbitrary $\tau\geq \aleph_0$, the following
implications are true:

$$
\begin{array}{cccccc} \text{strictly }a_{\tau}\text{-space}&\longrightarrow &
a_{\tau}\text{-space}  \\
\downarrow &     &\downarrow &  \\
\text{strictly }a_{\tau}\text{-subcompact space}&\longrightarrow &
a_{\tau}\text{-subcompact space}
\\
  \downarrow &  &\downarrow & \\
 \text{almost strictly }a_{\tau}\text{-subcompact space} &\longrightarrow & \text{almost
 }a_{\tau}\text{-subcompact space}
\end{array}
$$

\begin{center} Diagram 1.
\end{center}

Further, we prove the strictness of all the implications in the
Diagram 1.

\begin{proposition} Let $X$ be an (strictly)
$a_{\tau}$-subcompact space and $C\in [X]^{\leq\tau}$. Then
$X\setminus C$ is an (strictly) $a_{\tau}$-subcompact space.
\end{proposition}

In particular, a space $X\setminus C$ is an (strictly)
$a_{\tau}$-subcompact space  for any (strictly) $a_{\tau}$-space
$X$ and $C\in [X]^{\leq\tau}$.

\begin{proposition} There exists an (strictly) $a$-subcompact
space $X$ such that it is not homeomorphic to $Y\setminus C$ where
$Y$ is an (strictly) $a$-space and $C\in [Y]^{\leq\omega}$.
\end{proposition}

\begin{proof} In \cite{pytk}, it was proved that any
non-$\sigma$-compact Borel subset of a Polish space admits a
condensation onto a metric compact space. Consider any Borel
subset $X$ of a Polish space that has a Borel order higher than
the first. Since,  for any $S\in [X]^{\leq\omega}$, the space
$X\setminus S$ is Borel not $\sigma$-compact subset of a Polish
space , by Theorem 1 in \cite{pytk}, $X\setminus S$ admits a
condensation onto $\mathbb{I}^{\aleph_0}$. Hence, $X$ is an
(strictly) $a$-subcompact space.

Assume that $X$ is homeomorphic to $Y\setminus C$ where $Y$ is an
(strictly) $a$-space and $C\in [Y]^{\leq\omega}$. Then $Y\setminus
C$ is a $G_{\delta}$-set in $Y$ and, hence, $Y\setminus C$ is a
Polish space. It follows that $X$ is Polish and it has a Borel
order equal to $1$.
\end{proof}

\begin{theorem} There is an almost strictly $a_{\tau}$-subcompact
space which is not $a_{\tau}$-subcompact.
\end{theorem}

\begin{proof} Consider $X=D\bigoplus K$, where $D$ is a
discrete set of cardinality $\tau^+$ and $K$ is an extremally
disconnected compact space without isolated points of cardinality
$(2^{2^{\tau^+}})^+$. We prove that $X$ is almost strictly
$a_{\tau}$-subcompact. Let $E\in [X]^{\leq \tau}$. Put $A=E\cap D$
and $B=E\cap K$. Choose $C\subset D\setminus A$ such that $|C| =
|B|$. Denote $D_1=D\setminus (A\cup C)$. Let $d_0\in D_1$ and let
$D_1^*$ be a one-point compactification of $D_1\setminus \{d_0\}$,
where $d_0$ is not-isolated point of the compact space $D_1^*$.
Put $X_1=D_1^*\cup K$. Let $\psi$ be a bijection between $C$ and
$B$. Construct a condensation $\varphi: X\setminus (A\cup
B)\rightarrow D_1^*\cup K$ by the following rule

$$
\varphi(x)=
\begin{cases}
\psi(x), & x\in C,\\
x,       & x\in (K\setminus B)\cup D_1.
\end{cases}
$$

The continuous function $\varphi$ be extended to the continuous
function $f: X\rightarrow D_1^*\cup K$ where for fix $k\in K$ $$
f(x)=
\begin{cases}
\psi(x), & x\notin (A\cup B),\\
x,       & x\in B,\\
k,       & x\in A.
\end{cases}
$$
 Note that $X$ cannot be condensed onto an $a_{\tau}$- space.
 Suppose $f: X \rightarrow T$ is a condensation from the space $X$ onto an $a_{\tau}$-space $T$.
 Since $|\overline{f(D)}|\leq 2^{2^{\tau^+}}$, there is $W\subset T\setminus
 \overline{f(D)}$ such that $W$ is open-closed in $f(K)$ ($f(K)$ is homeomorphic to $K$).
Hence, $W$ is an extremally disconnected compact space without
isolated points. Let $S$ be a countable subset of $W$. We show
that $T\setminus S$ cannot be condensed onto a compact space.
Indeed, let $g: T\setminus S \rightarrow B$ be a condensation
where $B$ is compact. Then $T\setminus W$ is compact, hence,
$g(T\setminus W)$ is compact, and $B\setminus g(T\setminus W)$ is
locally compact. By Theorem in \cite{parh2}, the space $B\setminus
g(T\setminus W)$ (and, hence, $W\setminus S$) admits a
condensation onto a compactum. This contradicts of Theorem
\ref{iliad}.
\end{proof}

\section{Main results}

\begin{proposition}\label{pr21} Let $X$ admits a condensation onto a strictly
$a_{\tau}$-space for some $\tau\geq \aleph_0$. Then $X$ is
strictly $a_{\tau}$-subcompact.
\end{proposition}

\begin{proof} Let $f: X \rightarrow Y$ be  a condensation from
a space $X$ onto a strictly $a_{\tau}$-subcompact space $Y$.
 Take any $S\in [X]^{\leq\tau}$. Then there is a condensation $h:
Y\setminus f(S) \rightarrow K$ where $K$ is compact, such that $h$
can be extended to a continuous map $\widetilde{h}: Y \rightarrow
K$. Note that $\widetilde{h}\circ f$ is a continuous extension
over $X$ of the condensation $h\circ (f\upharpoonright (X\setminus
S)): X\setminus S \rightarrow K$.
\end{proof}

In 1970 S.Mr$\acute{o}$wka \cite{mrow} generalized the class of
dyadic spaces defining the class of polyadic spaces (= the
continuous images of the products of the one point
compactifications of discrete spaces).

In paper \cite{kt} W.Kulpa and M.Turza$\acute{n}$ski introduced
the class of weakly dyadic spaces.

Let $T$ be an infinite set. Denote a Cantor cube by

$D^T:=\{p: p:T\rightarrow \{0,1\}\}$. For $s\subset T$ and $p\in
D^T$ we shall use the following notation

$G_s(p):=\{f\in D^T: f\upharpoonright s= p\upharpoonright s$ and
$p^{-1}(0)\subset f^{-1}(0)\}$.

\begin{definition} (\cite{kt})

$\bullet$ A subset $X\subset D^T$ is said to be an {\it
$\omega$-set} iff for each $p\in X$ there exists an $s\subset T$
such that $|s|\leq \omega$ and $G_s(p)\subset X$.

$\bullet$ A space $Y$ is said to be {\it a weakly dyadic space} if
$Y$ is a continuous image of a compact $\omega$-set in $D^T$.

\end{definition}

The class of all weakly diadic spaces contains the class of all
centered spaces in sense of Bell \cite{bell} which in turn,
contains the class of all polyadic spaces. Kulpa and
Turza$\acute{n}$ski proved that a weakly dyadic space is a
strictly $a$-space (Lemma 2 and Theorem in \cite{kt}).



\begin{corollary} Suppose that $X$ admits a condensation
onto a weakly diadic space. Then $X$ is a strictly $a$-space.
\end{corollary}

\begin{theorem} Let $X=Z\bigoplus (\bigoplus \{X_{\alpha} :
\alpha\in A\})$, where $|A|=\tau$,  $X_{\alpha}$ is an $a_{\tau}$-
space, $|X_{\alpha}|>\tau$ for every $\alpha\in A$ and $Z$ is
compact. Then $X$ is almost $a_{\tau}$-subcompact.
\end{theorem}

\begin{proof} Let $S\in [X]^{\leq \tau}$. Without loss of
generality we can assume that $S\subseteq Z$. Otherwise, for every
$\alpha\in A$ there is a condensation $\varphi_{\alpha}$ from
$X_{\alpha}\setminus S$ onto a compactum. Then we can consider the
space $Z\bigoplus (\bigoplus \{\varphi(X_{\alpha}\setminus S) :
\alpha\in A\})$ where the restriction $\varphi\upharpoonright
(X_{\alpha}\setminus S)=\varphi_{\alpha}$ for every $\alpha\in A$.

Note that the space $\bigoplus \{X_{\alpha} : \alpha\in T\}$ is
locally compact for any $T\subset A$, and, by Parhomenko's Theorem
\cite{parh2}, it admits a condensation onto a compactum. Thus, if
$|S|<|A|$, then we can consider the space $P=Z\bigoplus (\bigoplus
\{X_{\alpha} : \alpha\in B\})$ where $B\subset A$ and $|B|=|S|$.
Let $S=\{s_{\alpha}: \alpha\in B\}$. Further, we suppose that
$X=P$ and $|S|=|A|$.

We prove that $X\setminus S$ is subcompact. For any $\alpha\in A$
we choose a point $p_{\alpha}\in X_{\alpha}$. On the set
$X\setminus S$ define the topology $\tau'$, topologize $X\setminus
S$ by letting sets:

1. $Vx=(Ox\cup (\cup \{X_{\alpha}: s_{\alpha}\in Ox\}))\cap
(X\setminus S)$ where $Ox$ is a neighborhood of $x$ in $X$ such
that $Ox\cap \{p_{\alpha}:\alpha\in A\}=\emptyset$, be basic
neighborhood of a point $x \in (X\setminus S)\setminus
\{p_{\alpha}: \alpha\in A\}$;

2. $Vp_{\alpha}=(Op_{\alpha}\cup Os_{\alpha})\cup(\cup
\{X_{\beta}: s_{\beta}\in Os_{\alpha}\setminus s_{\alpha}\})\cap
(X\setminus S)$ where $Op_{\alpha}\subseteq X_{\alpha}$ is a
neighborhood of $p_{\alpha}$ in $X$ and  $Os_{\alpha}$ is a
neighborhood of $s_{\alpha}$ in $X$, be basic neighborhood of a
point $p_{\alpha}$ for each $\alpha\in A$.

Note that $(X\setminus S, \tau')$ is a Hausdorff space and
$X\setminus S$ admits a condensation onto $(X\setminus S, \tau')$.
It remains to prove that $(X\setminus S, \tau')$ is compact.

We show that any infinite set $M\subseteq X\setminus S$ has a
complete accumulation point in topology $\tau'$.

Perhaps $|M\cap Z|=|M|$. The set $M$ has a complete accumulation
point $z$ in the compact space $Z$. If $z\notin S$, then $z$ is a
complete accumulation point of $M$ in $(X\setminus S, \tau')$; if
$z\in S$, then $z=s_{\alpha}$ for some $\alpha\in A$. Any basic
neighborhood of $p_{\alpha}$ includes some neighborhood
$Os_{\alpha}$ of $s_{\alpha}$ and, hence, the point $p_{\alpha}$
is a complete accumulation point of $M$ in the topology $\tau'$.

Further, we assume that $|M\cap Z|<|M|$.

 If there is an index $\alpha$ such that $|M\cap
X_{\alpha}|=|M|$, then a complete accumulation point $x$ of the
set $M\cap X_{\alpha}$ will be a complete accumulation point of
$M$ in  $(X\setminus S, \tau')$. If there is no such index, then
$|\{\alpha: M\cap X_{\alpha}\neq \emptyset \}|=|M|$. The set
$\{s_{\alpha}: M\cap X_{\alpha}\neq \emptyset\}$ has the same
cardinality as $M$. The set $\{s_{\alpha}: M\cap X_{\alpha}\neq
\emptyset\}$ has a complete accumulation point $x$ in $Z$. If
$x\in Z\setminus S$, then $x$ will be a complete accumulation
point of $M$ in $(X\setminus S, \tau')$ because $Vx$ includes all
$X_{\alpha}$ for $s_{\alpha}\in Ox$ for any neighborhood $Ox$ of
$x$. If $x=s_{\alpha_0}$, then the basic neighborhood
$Vp_{\alpha_0}$ includes all $X_{\alpha}$ for $s_{\alpha}\in
Os_{\alpha_0}$ for any neighborhood $Os_{\alpha_0}$  and, hence,
$p_{\alpha_0}$ is a complete accumulation point of the set $M$.
\end{proof}

Note that in the previous theorem we also proved the following
proposition.

\begin{proposition} Let $X=Z\bigoplus (\bigoplus \{X_{\alpha} :
\alpha\in A\})$, where  $X_{\alpha}$ is a non-empty compact space
for each $\alpha\in A$ , $|A|=\tau$, $Z$ is compact and $S\in
[Z]^{\leq \tau}$. Then $X\setminus S$ is subcompact.
\end{proposition}

\begin{definition} A space $X$ is called {\it almost
$a_{<\tau}$-subcompact}, if for any $C\in [X]^{<\tau}$ there is a
condensation of $X\setminus C$ onto a compactum.
\end{definition}

\begin{theorem} Let $X_{\alpha}$ be a non-empty compact space for any
$\alpha\in A$, where $|A|=\tau$. Then $X=\bigoplus \{X_{\alpha} :
\alpha\in A\}$ is almost $a_{<\tau}$-subcompact.
\end{theorem}

\begin{proof} Let $S\in [X]^{<\tau}$. Then the set of indexes
$B=\{\alpha\in A: X_{\alpha}\cap S=\emptyset\}$ has the
cardinality $\tau$. The set $X_1=\bigoplus \{X_{\alpha} :
\alpha\in A\setminus B\}$ is a locally compact space. By the
theorem of Parkhomenko \cite{parh2}, there is a condensation
$\varphi : X_1 \rightarrow Z$ from the space $X_1$ onto a compact
space $Z$. By Proposition 2.5, the space $Z\bigoplus (\bigoplus
\{X_{\alpha} : \alpha\in B\})\setminus \varphi(S)$ is subcompact.
\end{proof}

We recall the definition of $\sum_{\tau}$- product of
 spaces for $\tau\geq \aleph_0$. Let
$\{X_{\lambda} : \lambda\in \Lambda \}$ be a family of topological
spaces. Let  $X=\prod \{X_{\lambda} : \lambda\in \Lambda \}$ be
the Cartesian product with the Tychonoff topology. Take a point
$p=(p_{\lambda})_{\lambda\in \Lambda}\in X$. For each
$x=(x_{\lambda})_{\lambda\in \Lambda}\in X$, let
$Supp(x)=\{\lambda\in \Lambda : x_{\lambda}\neq p_{\lambda}\}$.
Then the subspace $\sum_{\tau}(p)=\{x\in X: |Supp(x)|\leq \tau\}$
of $X$ is called {\it $\sum_{\tau}$- product}  $\{X_{\lambda} :
\lambda\in \Lambda \}$ about $p$. The subspace $\{x\in X:
|Supp(x)|<\aleph_0\}$ of $X$ is called {\it $\sigma$- product}. In
1959, H.H. Corson \cite{cor} introduced the definitions of
$\sum_{\aleph_0}-$ products and $\sigma$-products and studied
these spaces.

\medskip

 Further, we study the property of the subcompactness of
$\sum_{\tau}$- ($\sigma$-) products of compacta for any $\tau\geq
\aleph_0$.

\begin{proposition} Let $X$ be $\sum_{\tau}$-product $\{X_{\beta}
: \beta\in B \}$ of non-trivial compacta $X_{\beta}$ of density at
most $\tau$ and $|B|\geq\tau^+$. Then $X$ is not subcompact.
\end{proposition}

\begin{proof} On the contrary, let $f: X\rightarrow K$ be a
condensation from $X$ onto a compact space $K$. Since $X$ is a
$\tau$- bounded space (the closure of any set of cardinality at
most $\tau$ is compact), then the space $K$ is a $\tau$- bounded
space.

 Consider two cases:

 (1) $d(K)>\tau$. Since $K$ is a $\tau$-bounded space, there is a family
$\{ K_{\alpha} : \alpha<\tau^+\}$ of compact subsets of $K$, such
that $K_{\alpha}\subset K_{\beta}$ for $\alpha<\beta$ and
$d(K_{\alpha})\leq \tau$. For each $K_{\alpha}$, there is a
compact subset $Y_{\alpha}$ of  $X$ such that
$f(Y_{\alpha})\supseteq K_{\alpha}$. Since $Y_{\alpha}$ depends on
$\tau$ coordinates in $\sum_{\tau}$-product $X$ and
$\alpha<\tau^+$, then $\bigcup\limits_{\alpha<\tau^+}
Y_{\alpha}\subset Y\subset X$, where $Y$ is $\sum_{\tau}$-product
of compacta $X_{\beta(\alpha)}$, where $Y_{\alpha} \subseteq
X_{\beta(\alpha)}$ and $d(X_{\beta(\alpha)})\leq \tau$ for each
$\alpha<\tau^+$.

 Note that  $\beta Y= \prod \{ X_{\beta(\alpha)} : \alpha<\tau^+\}$ where $\beta Y$ is the Stone-$\check{C}$ech compactification of the space $Y$
 \cite{ArchPonom}.
 Since $X_{\beta(\alpha)}$ is a compact space of density at most
$\tau$ for each $\alpha<\tau^+$, then, by the Hewitt-
Marczewski-Pondiczery theorem (Theorem 2.3.15 in \cite{enge}), the
space $\beta Y$ has a density at most $\tau$. The condensation
$f\upharpoonright Y: Y \rightarrow f(Y)$ can be extended to the
continuous function $h: \beta Y \rightarrow f(Y)$. It follows that
$f(Y)$ is of density at most $\tau$ and it contains an increasing
transfinite sequence $\{ K_{\alpha} : \alpha<\tau^+\}$.
Contradiction.

(2) $d(K)\leq\tau$. Let $S$ be a dense subset of the space $K$ and
$|S|\leq \tau$. Consider the subset $f^{-1}(S)$ of $X$. Since each
point $x\in f^{-1}(S)$ depends on $\tau$ coordinates $\alpha$,
then the set $f^{-1}(S)$ also depends on $\tau$ coordinates
$\{\alpha_s : s\in L\}$, $|L|\leq \tau$. Hence, $f^{-1}(S)\subset
Z=\prod \{ X_{\alpha_s}: s\in L\}$. Note that $Z$ is a compact
space of density at most $\tau$ and $f(Z)=K$. Since $X\neq Z$ and
$f$ is an injective mapping, we get a contradiction.

\end{proof}

\begin{corollary} An uncountable $\sum$-product of non-trivial
metrizable compacta is not subcompact.
\end{corollary}

{\bf Question 1.} Let $X$ be $\sum_{\tau}$-product $\{X_{\beta} :
\beta\in B \}$ of non-trivial compacta $X_{\beta}$ of density (or
weight) $\geq\tau^+$ for each $\beta\in B$ and $|B|\geq\tau^+$.
Will $X$ be a subcompact space?

\begin{theorem} Let $X$ be an infinite $\sigma$-product of
compacta. Then $X$ is not subcompact.
\end{theorem}

\begin{proof} Since $X$ is an infinite $\sigma$-product of
compacta, it is a countable union of compacta with an empty
interior ($X$ is a space of the first category) and, hence, $X$
cannot be condensed onto a compactum (a compact space is a space
of the second category).
\end{proof}

Note that the inverse limit of an inverse system of (strictly)
$a$-spaces may not be an $a$-space. Indeed, let $X$ be a compact
space that is not an $a$-space. Since $X$ is a subset of
$\mathbb{I}^{\alpha}$, where $\mathbb{I}=[0,1]$ and $\alpha=w(X)$,
the space $X$ can be represented as the inverse limit of an
inverse system of $\mathbb{S}=\{\mathbb{I}^{\beta},
\pi^{\sigma}_{\beta}, \omega_1\}$, where $\beta\leq\sigma$,
$\beta,\sigma\in \omega_1$, $\pi^{\sigma}_{\beta}: \pi_{\sigma}(X)
\rightarrow \pi_{\beta}(X)$ and $\pi_{\eta}:X \rightarrow
\mathbb{I}^{\eta}$ is the projection for $\eta\in \omega_1$. Note
that $\mathbb{I}^{\eta}$ is a metrizable compact space (an
$a$-space) for each $\eta\in \omega_1$, but
$X=\lim\limits_{\leftarrow} \mathbb{S}$ is not an $a$-space.

\medskip

The following theorem was proved in \cite{Pytkeev51}(Theorem 1).

\begin{theorem}\label{thp} Let $\{X_{\alpha} : \alpha\in
\Lambda\}$ be a family of non-empty Tychonoff spaces,
$w(X_{\alpha})\leq \tau$ for each $\alpha\in \Lambda$ and
$|\Lambda|=2^{\tau}$. Then $\bigoplus \{X_{\alpha} : \alpha\in
\Lambda\}$ admits a condensation onto $\mathbb{I}^{\tau}$.
\end{theorem}

Recall that the $i$-weight $iw(X)$ of a space $X$ is the smallest
infinite cardinal number $\kappa$ such that $X$ can be mapped by a
one-to-one continuous mapping onto a space of the weight not
greater than $\kappa$.

\begin{theorem} Let $\{ X_{\alpha}: \alpha\in \Lambda\}$ be a
family of non-empty spaces, $iw(X_{\alpha})\leq \tau$ for each
$\alpha\in \Lambda$ and $|\Lambda|=2^{\tau}$. Then $\bigoplus
\{X_{\alpha}: \alpha\in \Lambda\}$ is $a_{<2^\tau}$-subcompact.
\end{theorem}

\begin{proof} Let $S\in [X]^{<2^\tau}$ where $X=\bigoplus
\{X_{\alpha}: \alpha\in \Lambda\}$. Since $iw(X_{\alpha})\leq
\tau$, for each $\alpha\in \Lambda$, there is a condensation
$X_{\alpha}$ onto a Tychonoff space $Y_{\alpha}$ of the weight
$\leq \tau$. Let $Y=\bigoplus \{Y_{\alpha}: \alpha\in \Lambda\}$.
Then there is a condensation $f: X \rightarrow Y$. Since
$w(Y_{\alpha}\setminus f(S))\leq \tau$ for each $\alpha\in
\Lambda$ and $|\{\alpha : Y_{\alpha}\setminus f(S)\neq \emptyset
\}|=2^{\tau}$, then, by Theorem \ref{thp}, $Y\setminus f(S)$
admits a condensation onto $\mathbb{I}^{\tau}$.
\end{proof}

Recall that the absolute $aX$ of a topological space $X$ is the
set of converging ultrafilters and the natural map $\pi_X: aX
\rightarrow X$ simply assigns the limit to each ultrafilter
\cite{powo}.

\begin{example} There is a compact space $X$ with
$\chi(x,X)>\tau$ (or $\pi\chi(x,X)>\tau$) for each $x\in X$ such
that it is not $a_{\tau}$-space.
\end{example}

\begin{proof} Let $Z$ be a compact space such that
$\chi(z,Z)>\tau$ ($\pi\chi(z,Z)>\tau$) for each $z\in Z$. Let
$X=aZ$ be the absolute of $Z$. Then $X$ is an extremally
disconnected compact space and $\chi(x,X)>\tau$
($\pi\chi(x,X)>\tau$) for each $x\in X$. By Theorem \ref{iliad},
$X$ is not $a_{\tau}$-space. Indeed, it is sufficient to consider
$X\setminus S$, where $S\in [X]^{\leq \tau}$.
\end{proof}

\begin{example} There are a compact non-$a_{\tau}$-space $X$
and an irreducible continuous mapping of $X$ onto
$\mathbb{I}^{\tau}$.
\end{example}

\begin{proof} Let $X=a\mathbb{I}^{\tau}$ be the absolute of
the space $\mathbb{I}^{\tau}$. The natural map
$\pi_{\mathbb{I}^{\tau}}: X \rightarrow \mathbb{I}^{\tau}$ is an
irreducible continuous surjection. Then, by Theorem \ref{iliad},
the space $X\setminus C$ where $C\in [X]^{\omega}$ is not
subcompact.
\end{proof}

\begin{theorem}\label{214} Let $X= \bigoplus \{X_{\alpha} : \alpha\in
\Lambda\}\bigcup \{\xi\}$ where $X_{\alpha}$ is a compact space
for each $\alpha\in \Lambda$ and a basic neighborhood of $\xi$
contains of all but finitely many sets $X_{\alpha}$. The space $X$
is an (strongly) $a_{\tau}$-space if and only if $X_{\alpha}$ is
an (strongly) $a_{\tau}$-space for each $\alpha\in \Lambda$.
\end{theorem}

\begin{proof} ($\Rightarrow$) Let $X$ be a (strictly)
$a_{\tau}$-space for some cardinal number $\tau$, $\beta\in
\Lambda$ and $S\in [X_{\beta}]^{\leq\tau}$. Since $X$ is a
(strictly) $a_{\tau}$-space,  there is a condensation
$f:X\setminus S \rightarrow Z$ from $X\setminus S$ onto a compact
space $Z$. Then the compact space $X\setminus X_{\beta}$ is
homeomorphic to the compact space $f(X\setminus X_{\beta})$. It
follows that the space
 $f(X_{\beta}\setminus S)=Z\setminus f(X\setminus X_{\beta})$ is
 locally compact and, by Parhomenko's Theorem \cite{parh2}, it
 admits a condensation onto a compact space.

($\Leftarrow$) Let  $X_{\alpha}$ be an (strictly) $a_{\tau}$-space
for each $\alpha\in \Lambda$ and some cardinal number $\tau$. Let
$S\in [X]^{\leq \tau}$. For each $\alpha\in \Lambda$ there exists
a condensation $f_{\alpha}: X_{\alpha}\setminus S \rightarrow
K_{\alpha}$, where $K_{\alpha}$ is compact.

1. $\xi\in S$. Then $X\setminus S$ admits a condensation $f$ onto
a locally compact space $Y=\bigoplus \{ K_{\alpha} : \alpha\in
\Lambda\}$ so that $f\upharpoonright (X_{\alpha}\setminus
S)=f_{\alpha}$ for each $\alpha\in \Lambda$. By Parhomenko's
Theorem \cite{parh2}, $Y$ admits a condensation onto a compactum.

2.  $\xi\notin S$. Then $X\setminus S$ admits a condensation $f$
onto a compact space $Z=\bigoplus \{ K_{\alpha} : \alpha\in
\Lambda\}\bigoplus \{p\}$ so that $f\upharpoonright
(X_{\alpha}\setminus S)=f_{\alpha}$ for each $\alpha\in \Lambda$
and $f(\xi)=p$, where a basic neighbourhood of the point $p$
contains of all but finitely many sets $K_{\alpha}$.
\end{proof}

\begin{remark} The previous theorem is not true in the class
of (strictly) $a_{\tau}$-subcompact spaces. In \cite{Pytkeev3}, it
was considered a space $X$ such that $X \bigoplus X$ admits a
consensation onto a metrizable compact space (moreover, $X
\bigoplus X$ is strictly $a_{\tau}$-subcompact), but $X$ is not
subcompact.
\end{remark}

\begin{theorem} Let $X= \bigoplus \{X_{\alpha} : \alpha\in
\Lambda\}\bigcup \{\xi\}$, where $X_{\alpha}$ is an (strictly)
$a_{\tau}$-space for each $\alpha\in \Lambda$, $\xi\notin \bigcup
\{X_{\alpha} : \alpha\in \Lambda\}$ and  $X$ is Hausdorff. Then
$X$ is an (strictly) $a_{\tau}$-space.
\end{theorem}

\begin{proof} Note that the space $X$ admits a condensation
onto $Y=\bigoplus \{X_{\alpha} : \alpha\in \Lambda\}\bigcup
\{\xi\}$, where a basic neighborhood of $\xi$ contains of all but
finitely many sets $X_{\alpha}$. By Theorem \ref{214} and
Proposition \ref{pr21}, the space $Y$ is an (strictly)
$a_{\tau}$-space.
\end{proof}

In \cite{bop} (Theorem 13), it is proved the following theorem.

\begin{theorem} Let $X=\prod\{X_{\alpha} : \alpha<\tau\}$ be
product of non-trivial compacta. Let $f: X\rightarrow Y$ be a
continuous surjection such that the cardinality $|f\pi^{-1}_A\pi_A
x|>\aleph_0$ for any $A\in [\tau]^{<\omega}$ and $x\in X$. Then
$Y$ is a strictly $a$-space.
\end{theorem}

An amplification of this theorem is the replacement of the
condition $|f\pi^{-1}_A\pi_A x|>\aleph_0$ with the condition
$f\pi^{-1}_A\pi_A x\neq fx$.

\begin{theorem} Let $X=\prod\{X_{\alpha} : \alpha<\tau\}$ be
 product of an infinite number of non-trivial compacta. Let $f: X\rightarrow Y$ be a
continuous surjection such that  $f\pi^{-1}_A\pi_A x\neq fx$ for
any $A\in [\tau]^{<\omega}$ and $x\in X$. Then $Y$ is a strictly
$a$-space.
\end{theorem}

\begin{proof} Let $x\in X$ and $A\in [\tau]^{<\omega}$. We
show that the set $f\pi^{-1}_A\pi_A x$ is uncountable. On the
contrary, suppose that  $f\pi^{-1}_A\pi_A x$ is countable. Then
the compact space $P=\pi_A^{-1}\pi_Ax=\bigcup
\{\pi_A^{-1}\pi_Ax\cap f^{-1}y: y\in f\pi_A^{-1}\pi_Ax\}$ is the
sum of countable collection of compacta $\pi_A^{-1}\pi_Ax\cap
f^{-1}y$. So there is $y\in f\pi_A^{-1}\pi_Ax$ for which
$Int_P(\pi_A^{-1}\pi_Ax\cap f^{-1}y)\neq\emptyset$. There is a
basic open set $U=\prod \{U_{\alpha}: \alpha\in B\}\times \prod
\{X_{\alpha}: \alpha\in \tau\setminus B\}$ where $B$ is finite and
$U_{\alpha}$ is open for all $\alpha\in B$. $U\cap
\pi_A^{-1}\pi_Ax\neq \emptyset$, $f(U\cap \pi_A^{-1}\pi_Ax)=y$.
Then $f\pi_{A\cup B}^{-1}\pi_{A\cup B}x'=y=fx'$ for any point
$x'\in U\cap \pi_A^{-1}\pi_Ax$, which contradicts the condition of
the theorem. We proved that $|f\pi^{-1}_A\pi_A x|>\aleph_0$ and,
thus, all the conditions of Theorem 2.18 are met.
\end{proof}

\begin{corollary} Let $X$ be dyadic compact and
$\chi(x,X)>\mathfrak{m}$ for each $x\in X$. Then $X$ is a strictly
$a_{\mathfrak{m}}$-space.
\end{corollary}

\begin{proof} By Theorem 2.18 and Theorem 11 in \cite{bop}.
\end{proof}

In  (\cite{bop}, Theorem 7), it is proved that a product
$X=\prod\{X_{\alpha} : \alpha<\tau\}$ of non-trivial compacta
$X_{\alpha}$ is a strictly $a_{\tau}$-space.

\medskip

If we require that every $X_{\alpha}$ be a subcompact space, we
get the following proposition.

\begin{proposition} The product $X=\prod\{X_{\alpha} :
\alpha<\tau\}$ of non-trivial subcompact spaces $X_{\alpha}$ is
strictly $a_{\tau}$-subcompact.
\end{proposition}

\begin{proof} Since $X_{\alpha}$ admits a condensation onto
a compact space $K_{\alpha}$ for each $\alpha<\tau$, it is
sufficient to note that $X$ admits a condensation onto the compact
space $K=\prod\{K_{\alpha} : \alpha<\tau\}$. By Theorem 7 in
\cite{bop}, $K$ is a strictly $a_{\tau}$-space. By Proposition
\ref{pr21}, $X$ is a strictly $a_{\tau}$-subcompact space.
\end{proof}

In particular, we obtain that an infinite product of non-trivial
subcompacts is a strictly $a$-subcompact space.

\medskip

Note that a continuous image of an infinite product of non-trivial
subcompact spaces may be not subcompact.

\medskip

\begin{example} Let $X=\mathbb{N}\times \prod \{D_{\alpha} :
\alpha<\omega_1\}$, where $D_{\alpha}=\{0,1\}$ is discrete for
each $\alpha<\omega_1$.
\end{example}

Consider a condensation $f: X \rightarrow Y$ where
$Y=\mathbb{Q}\times \prod \{D_{\alpha} : \alpha<\omega_1\}$. The
space $Y$ can be represented in the form: $Y=\bigcup \{
\{p\}\times C : p\in \mathbb{Q} \}$ where $C=\prod \{D_{\alpha} :
\alpha<\omega_1\}$ is the Cantor cube. Thus, $Y$ is the countable
sum of compacta with an empty interior ($Y$ is a space of the
first category) and, hence, $Y$ cannot be condensed onto a
compactum (a compact space is a space of the second category).

\medskip

 In (\cite{bop}, Theorem 5), it is proved that the product of a compact space and a metrizable compact space without isolated points is
 a strictly $a$- space.

\begin{proposition} If $X$ admits a condensation onto a metrizable
compact space without isolated points and $Y$ is subcompact, then
$X\times Y$ is a strictly $a$-subcompact space.
\end{proposition}

\begin{proof} Let $\varphi: X\rightarrow X_1$ be a
condensation from $X$ onto a metrizable compact space $X_1$
without isolated points and suppose $\psi: Y \rightarrow Y_1$ is a
condensation from $Y$ onto a compact space $Y_1$. $f:X\times Y
\rightarrow X_1\times Y_1$ where $f(x,y)=(\varphi(x),\psi(y))$ is
a condensation from $X\times Y$ onto a compact space $X_1\times
Y_1$ without isolated points. By Theorem 5 in  \cite{bop},
$X_1\times Y_1$ is a strictly $a$-space. By Proposition
\ref{pr21}, $X\times Y$ is a strictly $a$-subcompact space.
\end{proof}

{\bf Question 2.} Will the product of a (metrizable) compact space
and a strictly $a_{\tau}$-subcompact space be a strictly
$a_{\tau}$-subcompact space?

\medskip

 In \cite{bop} (Proposition 4), it is proved the following result.

\begin{proposition} There are ordered compacta $X$ and $Y$ such
that $X\times Y$ is not an $a$-space.
\end{proposition}

{\bf Question 3.} Suppose that a strictly $a_{\tau}$-space $Z$ is
represented as $Z=X\times Y$, where $\tau\geq\aleph_0$. Will at
least one of the multipliers be a strictly  $a_{\tau}$-space?

\medskip
\textbf{Acknowledgement} The authors would like to thank the
referee for careful reading and valuable comments and suggestions.
The work was performed as part of research conducted in the Ural
Mathematical Center.

\bibliographystyle{model1a-num-names}
\bibliography{<your-bib-database>}

\begin{thebibliography}{10}

\bibitem{alex}
P.S. Aleksandrov, \textit{On some basic directions in general
topology}, Uspekhi Mat. Nauk, {\bf 19}:6, (120), (1964), 3--46 (in
Russian); P.S. Aleksandrov, \textit{On some basic directions in
general topology}, Russian Math. Surveys, {\bf 19}:6, (1964),
1--39 (in English). Erratum. P.S. Aleksandrov, Corrections to the
article \textit{On some basic directions in general topology},
Uspekhi Mat. Nauk, 1965, {\bf 20}:1, (121), 253--254.


\bibitem{AlexUrys}
P.S. Aleksandrov and P.S. Uryson, \textit{Memoir on compact
topological spaces}, 3rd ed. Nauka, Moscow., 1971; MR, 51 1951,
6719; See also P.S. Aleksandrov, "On compact topological spaces",
Works on topology and other fields of mathematics, 2, P.S. Uryson,
GITTL, Moscow; P.S. Aleksandrov and P.S. Uryson, MR, 14-12.

\bibitem{alexhopf}
P.S. Alexandrov, H. Hopf, \textit{Topologie I}, Berlin, 1935.

\bibitem{arh1}
A.V. Arhangel'skii, \textit{On condensations of $C_p$-spaces onto
compacta}. Proc. Am. Math. Soc., {\bf 128}, (2000), 1881-1883.

\bibitem{arh2}
A.V. Arhangel'skii, O. Pavlov, \textit{A note on condensations of
$C_p(X)$ onto compacta}. Comment. Math. Univ. Carol., {\bf 43},
(2002), 485--492.






\bibitem{ArchPonom}
A.V. Arhangel'skii, V.I. Ponomarev, \textit{Basics of general
topology in problems and exercises}, M. Science, 1974, 424 p. (in
Russian). (A.V. Arkhangel'skii, V.I. Ponomarev,
\textit{Fundamentals of General Topology. Problems and Exercises}.
Mathematics and Its Applications {\bf 13}, Springer Netherlands,
1984).







\bibitem{banach}
S. Banach, \textit{Problem 26}, Colloq. Math. {\bf 1}, (1948),
150.


\bibitem{bell}
M.G. Bell, \textit{Generalized dyadic spaces}, Fundamenta
Mathematicae CXXXV (1985), 47--58.

\bibitem{bel1}
 V.K. Bel'nov, \textit{Compressions onto compacta}, Dokl. Akad. Nauk SSSR, {\bf 193}:3, (1970),
 506-509. (in Russian)



\bibitem{Belug0}
V.I. Belugin, \textit{Contractions onto bicompact}, Dokl.Akad.
Nauk SSSR, {\bf 207}:2, (1972), 259--261.(in Russian)

\bibitem{Belug1}
V.I. Belugin, \textit{Contractions onto bicompact}, Dokl. Akad.
Nauk Bolg., {\bf28}:11, (1975), 1447--1449.




\bibitem{bop}
V.I. Belugin, A.V. Osipov, E.G. Pytkeev, \textit{On classes of
subcompact spaces}, Math. Notes.


\bibitem{chjm}
W.W. Comfort, A.W. Hager, J. van Mill, \textit{Compact
condensations and compactifications}. Topology and its
Applications. {\bf 259}, (2019), 67--79.

\bibitem{cor}
H.H. Corson, \textit{Normality in subsets of product spaces}.
Amer. J. Math., {\bf 81}, (1959), 785--796.






\bibitem{enge}
R. Engelking, \textit{General Topology}, PWN, Warsaw, (1977); Mir,
Moscow, (1986).





\bibitem{kat}
M. Katetov, \textit{$H$-closed extensions of topological spaces}.
$\check{C}$asop. Math., fys., {\bf 69}, (1940), 36--39.


\bibitem{kt}
W. Kulpa, M. Turza$\acute{n}$ski, \textit{Bijections onto compact
spaces}. Acta Universitatis Carolinae. Mathematica et Physica,
{\bf 29}:2 (1988), 43--49.


\bibitem{iliad}
S. Iliadis, \textit{Some properties of absolutes}, Dokl. Akad.
Nauk SSSR, {\bf 152}:4 (1963), 798--800. (in Russian)


\bibitem{mor}
R.L. Moore, \textit{Foundations of point set theory}, AMS
Colloquium Publ. XIII, N.Y. 1932.



\bibitem{mrow}
S. Mr$\acute{o}$wka, \textit{Mazur theorem and $m$-adic spaces}.
Bull. Acad. Polon. Sci. S$\acute{e}$r. Sci. Math. Astronom. Phys.
{\bf 18}, (1970), 299--305.

\bibitem{osipyt}
A.V. Osipov, E.G. Pytkeev, \textit{On the problem of condensation
onto compacta}, Dokl. Akad. Nauk, {\bf 488:2}, (2019), 130--132.


\bibitem{parh1}
 A.S. Parhomenko, \textit{\"{U}ber eineindeutige stetige Abbildungen}, Rec. Math. [Mat. Sbornik] N.S., {\bf 5(47)}:1, (1939), 197--210.

\bibitem{parh2}
A.S. Parhomenko, \textit{\"{U}ber eineindeutige stetige
Abbildungen auf kompakte Raume}, Izv. Akad. Nauk SSSR Ser. Mat.,
{\bf 5}:3 (1941), 225-232.


\bibitem{powo}
J. R. Porter, R. G. Woods, \textit{Extensions and absolutes of
Hausdorff spaces}, Springer-Verlag, 1988.



\bibitem{proiz}
V.V. Proizvolov, \textit{On one-to-one continuous mappings of
topological spaces}, Mat. Sb. (N.S.), {\bf 68(110)}:3 (1965),
417--431.



\bibitem{pytk}
E.G. Pytkeev, \textit{Upper bounds of topologies}, Math. Notes,
{\bf 20}:4 (1976), 831-837.



\bibitem{Pytkeev3}
E.G. Pytkeev, \textit{On the theory of condensations onto compact
metric spaces}, Dokl. Akad. Nauk SSSR, {\bf 233}:6 (1977),
1046--1048.


\bibitem{Pytkeev1}
E.G. Pytkeev, \textit{Hereditarily plumed spaces}, Math. Notes,
{\bf 28}:4 (1980), 761--769.


\bibitem{Pytkeev2}
E.G. Pytkeev, \textit{On condensations onto compact Hausdorff
spaces}, Dokl. Akad. Nauk SSSR, {\bf 265}:4 (1982), 819--823.


\bibitem{Pytkeev51}
E.G. Pytkeev, \textit{On the theory of one-to-one continuous
mappings}. Research on modern analysis. Sverdlovsk. UrGU. Math.
Notes. {\bf 10}:2, (1977), 121--132.

\bibitem{raut}
I.L. Raukhvarger, \textit{On condensations into compacts}. Dokl.
Akad. Nauk SSSR,, {\bf 66}:13, (1949), 13--15.

\bibitem{rit}
H. Reiter, \textit{Spaces with compact subtopologies}. Rocky Mt.
J. Math., {\bf 2}, (1972), 239--247.

\bibitem{Serp}
W. Sierpi$\acute{n}$ski, \text{Un th$\acute{e}$or$\grave{e}$me sur
les continus}. Tohoru Math. J., {\bf 13}:3, (1918), 300--303.


\bibitem{smir}
Y.M. Smirnov, \textit{Condensations onto bicompacts and connection
with bicompact extensions and with retraction}. Fundam. Math.,
{\bf 63}:2, (1968), 199--211. (in Russian)

\bibitem{Fedor}
V.V. Fedorchuk, \textit{Bicompacta with noncoinciding
dimensionalities}, Dokl. Akad. Nauk SSSR, {\bf 182}:2 (1968),
275--277.


\bibitem{hedg1}
N. Hadzhiivanov, \textit{Extension of mappings into spheres and P.
S. Aleksandrov's problem of bicompact compressions}, Dokl. Akad.
Nauk SSSR, {\bf 194}:3 (1970), 525--527.





\bibitem{enc}
K.P.Hart, Jun-iti Nagata, J.E.Vaughan, \textit{Encyclopedia of
General Topology}, Elsevier Science, 2003, 536 p. (Y.Tanaka,
\textit{b-4- Quotient Spaces and Decompositions}), 43--46.











\end{thebibliography}







\end{document}